\documentclass{amsart}

\usepackage{amssymb}
\usepackage{color}
\usepackage{tikz-cd}

\newtheorem{thm}{Theorem}[section]

\newtheorem{cor}[thm]{Corollary}
\newtheorem{prop}[thm]{Proposition}
\newtheorem{thm*}{Theorem}

\theoremstyle{definition}
\newtheorem{definition}[thm]{Definition}

\theoremstyle{remark}
\newtheorem{remark}[thm]{Remark}

\numberwithin{equation}{section}

\def\mrm#1{{\mathrm{#1}}}

\newcommand{\brat}[1]{{\left< #1 \right>}}

\newcommand{\R}{{\mathbb{R}}}
\newcommand{\Z}{{\mathbb{Z}}}
\newcommand{\C}{{\mathbb{C}}}
\newcommand{\Q}{{\mathbb{Q}}}

\newcommand{\D}{{\mathbb{D}}}
\newcommand{\bK}{{\mathbb{K}}}

\newcommand{\ra}{\rightarrow}
\newcommand{\del}{\partial}

\newcommand\sign{\operatorname{sign}}

\newcommand{\cL}{\mathcal{L}}

\newcommand{\til}[1]{\widetilde{#1}}

\newcommand{\ol}[1]{\overline{#1}}

\newcommand{\om}{\omega}
\newcommand{\al}{\alpha}
\newcommand{\la}{\lambda}

\newcommand{\ga}{\gamma}
\newcommand{\Ga}{\Gamma}

\newcommand{\cA}{\mathcal{A}}

\newcommand{\cI}{\mathcal{I}}

\newcommand{\cP}{\mathcal{P}}

\newcommand{\rT}{\mathrm{T}}
\newcommand{\rS}{\mathrm{S}}
\newcommand{\fix}{\mathrm{Fix}}
\newcommand{\crit}{\mathrm{Crit}}

\DeclareMathOperator{\rank}{\mathrm{rank}}

\DeclareMathOperator{\Ham}{\mathrm{Ham}}

\DeclareMathOperator{\Symp}{\mathrm{Symp}}

\DeclareMathOperator{\im}{\mathrm{im}}

\DeclareMathOperator{\id}{\mathrm{id}}

\DeclareMathOperator{\cn}{\mathrm{CN}}
\DeclareMathOperator{\hn}{\mathrm{HN}}
\DeclareMathOperator{\cfn}{\mathrm{CFN}}
\DeclareMathOperator{\hfn}{\mathrm{HFN}}
\DeclareMathOperator{\cf}{\mathrm{CF}}
\DeclareMathOperator{\hf}{\mathrm{HF}}
\DeclareMathOperator{\h}{\mathrm{H}}
\DeclareMathOperator{\flux}{\mathrm{Flux}}
\DeclareMathOperator{\wflux}{\mathrm{\widetilde{Flux}}}
\DeclareMathOperator{\cz}{\mathrm{CZ}}

\begin{document}

\title{Remarks on symplectic circle actions, torsion and loops}


\author{Marcelo S. Atallah}

%
%


\begin{abstract}
We study loops of symplectic diffeomorphisms of closed symplectic manifolds. Our main result, which is valid for a large class of symplectic manifolds, shows that the flux of a symplectic loop vanishes whenever its orbits are contractible. As a consequence, we obtain a new vanishing result for the flux group and new instances where the presence of a fixed point of a symplectic circle action is a sufficient condition for it to be Hamiltonian. We also obtain applications to symplectic torsion, more precisely, non-trivial elements of $\Symp_{0}(M,\om)$ that have finite order.
\end{abstract}

\maketitle

\section{Introduction}
\subsection{The Flux group and the $e$-homomorphism}\label{sec: e-homomorphism}
Let $(M,\om)$ be a closed symplectic manifold and denote by $\til\Symp_{0}(M,\om)$ the universal cover of the identity component $\Symp_{0}(M,\om)$ of the group of symplectic diffeomorphisms. The \textit{flux homomorphism}
\begin{equation*}
	\wflux:\til\Symp_{0}(M,\om)\ra\h^{1}(M;\R)
\end{equation*}
 is defined by assigning to each class $\til\psi\in\til\Symp_{0}(M,\om)$ a cohomology class
 \begin{equation*}
 	\wflux(\til\psi)=\int^{1}_{0}[\iota_{X^{t}}\om]dt,
 \end{equation*}
where $X^{t}$ is the time-dependent vector field induced by a symplectic path $\{\psi_{t}\}$ representing $\til\psi$. In particular, if $\ga$ is a $1$-cycle in $M$, we have that
\begin{equation}\label{eq: Flux_evaluation}
	\brat{\wflux(\{\psi_{t}\}),[\ga]}=\int_{[0,1]\times\R/\Z}\al^{*}\om,
\end{equation}
where $\al:[0,1]\times\R/\Z\rightarrow M$ is given by setting $\al(t,s)=\psi_{t}(\ga(s))$. We shall often denote by $\wflux(\{\psi_{t}\})$ the flux of $\til\psi$. The \textit{flux group} $\Ga$ is the image of the restriction of the flux homomorphism to $\pi_{1}(\Symp_{0}(M,\om))\subset\til\Symp_{0}(M,\om)$. The $\wflux$ map descends to a homomorphism
 \begin{equation*}
 	\flux:\Symp_{0}(M,\om)\ra\h^{1}(M;\R)/\Ga,
 \end{equation*}
whose kernel was shown by Banyaga \cite{Banyaga} to be equal to $\Ham(M,\om)$. In particular, we have the exact sequence
\begin{equation*}
1\ra\Ham(M,\om)\ra\Symp_{0}(M,\om)\ra\h^{1}(M;\R)/\Ga\ra1,
\end{equation*}
which implies that $\Ham(M,\om)$ coincides with $\Symp_{0}(M,\om)$ if and only if $\h^{1}(M;\R)$ vanishes. For a base point $x_{0}\in M$, denote by
\begin{equation*}
	ev:\pi_{1}(\Symp_{0}(M,\om),\id)\ra\pi_{1}(M,x_{0})
\end{equation*}
the \textit{evaluation homomorphism} given by setting $ev([\{\psi_{t}\}])=[\{\psi_{t}(x_{0})\}]$. The image of $ev$ lies in the center of $\pi_{1}(M)$ (see \cite{Polterovich2006floer}). Therefore, since $M$ is connected, the evaluation maps for different choices of base points $x_{0}\in M$ are identified. Hence, we denote $ev$ without making specific reference to $x_{0}$. Observe that the evaluation map factors through the flux group $\Ga$ yielding the following commutative diagram:\footnote{The fact that $ev$ factors through the flux group was pointed out to me by Egor Shelukhin.}
\begin{equation}\label{dgm: evaluation_flux}
    \begin{tikzcd}
        \pi_{1}(\Symp_0(M,\om)) \arrow[d, "\wflux"] \arrow[rd, "ev"] \\
        \Ga\arrow[r,"e"] & \pi_{1}(M).
    \end{tikzcd}
\end{equation}
To see that the \textit{$e$-homomorphism} is well-defined note that if \[\wflux(\{\phi_{t}\})=\wflux(\{\psi_{t}\}),\] then the loop $\{\phi_{t}\circ\psi_{t}^{-1}\}$ of symplectic diffeomorphisms can be homotoped to a Hamiltonian loop, which is known to have contractible orbits (see \cite{McDuff2004survey}). An important consequence of diagram (\ref{dgm: evaluation_flux}) is that whenever $e$ is injective a symplectic loop with contractible orbits can be homotoped to a Hamiltonian loop. It is not hard to see that this is the case for $(M,\om)$ symplectically aspherical: $[\om]$ vanishes on the image $\h^{S}_{2}(M;\Z)$ of the Hurewicz map $\pi_{2}(M)\rightarrow \h_{2}(M;\Z)$. Indeed, if a symplectic loop $\{\psi_{t}\}$ has trivial evaluation and $\gamma$ is any $1$-cylcle in $M$, then the torus $\al(\rT^{2})$, where $\al(t,s)=\psi_{t}(\gamma(s))$, will have the symplectic area of a sphere. Equation (\ref{eq: Flux_evaluation}) then implies that the flux of the loop $\{\psi_{t}\}$ vanishes.\\ 

\noindent The injectivity of the $e$-homomorphism has important consequences in the theory of symplectic circle actions discussed in Section \ref{sec: Symplectic_circle_actions}. Furthermore, it allows one to deduce the vanishing of the Flux group in cases where the evaluation map is trivial. In particular, we have the following:
\begin{prop}\label{prop: vanishing_of_flux_group}
Let $(M,\om)$ be a closed symplectic manifold with injective $e$-homomorphism, further suppose that either $\chi(M)\neq0$ or that $\pi_{1}(M)$ has finite center. Then the flux group $\Ga$ is trivial.
\end{prop}
\begin{proof}
When $\chi(M)\neq0$ we have that $ev$ is trivial, a fact that is true even in the more general setting of loops of diffeomorphisms on a closed manifold, see \cite{Le-Ono_Ukraine}. Therefore, $\Gamma$ is trivial by the injectivity of $e$. Now, suppose that $\pi_{1}(M)$ has finite center and let $\{\psi_{t}\}$ be a symplectic loop. Then, there exists some positive integer $k$ such that $ev(\{\psi_{t}\})^{k}=1$. Thus, we have the following:
	\begin{equation*}
		e(\wflux(\{\psi^{k}_{t}\})) = ev(\{\psi^{k}_{t}\})=ev(\{\psi_{t}\})^{k}=1.
	\end{equation*}
We conclude, by the injectivity of $e$ and the fact that $\wflux$ is a homomorphism to a torsion free group, that the loop $\{\psi_{t}\}$ has no flux.
\end{proof}

\noindent Our main result provides more sufficient conditions for the injectivity of the $e$-homomorphism. 
\pagebreak
\begin{definition}
Let $(M,\om)$ be a closed symplectic manifold of dimension $2n$. We say that it satisfies condition $(\bigstar)$ if at least one of the following is true:
\begin{itemize}
	\item \textit{Symplectically aspherical}: The cohomology class $[\om]$ vanishes on the image $\h^{S}_{2}(M;\Z)$ of the Hurewicz map $\pi_{2}(M)\rightarrow \h_{2}(M;\Z)$.
	\item \textit{Spherically monotone}: There exists a constant $\la\in\R\setminus\{0\}$ such that \[[\om]|_{\pi_{2}(M)}=\lambda\cdot c_{1}(M)|_{\pi_{2}(M)},\] where $c_{1}(M)$ denotes the first Chern class associated with $(M,\om)$. We say \textit{positive (resp. negative) spherically monotone} when $\lambda>0$ (resp. $\lambda<0$).
	\item \textit{Weak Lefschetz property:} The multiplication map \[[\om]^{n-1}:\h^{1}(M;\R)\rightarrow\h^{2n-1}(M;\R)\] is injective (hence an isomorphism).
\end{itemize}
\end{definition}
\begin{remark}[Examples satisying $(\bigstar)$]
The standard symplectic torus $(T^{2n}, dp\wedge dq)$ and any closed oriented surface $\Sigma_{g}$ of genus $g\geq1$ with the standard area form are symplectically aspherical. Symplectic products of the form $(M\times N, \om_{M}\oplus\om_{N})$, where $(M,\om_{M})$ is positive (resp. negative) homologically monotone and $(N,\om_{N})$ is aspherical, are positive (resp. negative) spherically monotone but not homologically monotone. In particular, $\C P^{n}\times\rT^{2m}$ with symplectic form $\om_{\mrm{FS}}\oplus dp\wedge dq$ is positive spherically monotone. Hypersurfaces of $\C P^{n}$ defined by setting $z_{0}^{m}+\cdots+z_{n}^m=0$, $m>n+1$, are negative homologically monotone (see \cite{McDuffSalamonJhol}). Finally, all closed K\"{a}hler manifolds satisfy the weak Lefschetz property.
\end{remark}
\begin{thm}\label{thm: main}
Let $(M,\om)$ be a closed symplectic manifold satisfying $(\bigstar)$, then the $e$-homomorphism is injective.
\end{thm}

\noindent When the weak Lefschetz property is satisfied the injectivity of the $e$-homomorphism follows from classical arguments that were known at least as early as the work of Ono in \cite{Ono_Lefschetz}, see also \cite{McDuffSalamonIntro3, LMP1}. Indeed, if $\{\psi_{t}\}$ is a symplectic loop inducing a symplectic vector field $X_{t}$, then the homology classes of its orbits are Poincar\'{e} dual to the class 
\begin{equation*}
	\frac{1}{\mrm{Vol}(M)}\bigg[\wflux(\{\psi_{t}\})\wedge\frac{\om^{n-1}}{(n-1)!}\bigg]\in\h^{2n-1}(M;\R).
\end{equation*}
Therefore, by the injectivity of $[\om]^{n-1}$, we have that $\wflux(\{\psi_{t}\})=0$ whenever $ev(\{\psi_{t}\})=1$, since every orbit of $\{\psi_{t}\}$ is homologically trivial. We also note that in \cite{Kedra-Kotschick-Morita}, the implication of Proposition \ref{prop: vanishing_of_flux_group} is proven under the weak Lefschetz assumption. 
\subsubsection{Outline of the proof of Theorem \ref{thm: main}}
When $(M,\om)$ is spherically monotone we provide two proofs of the injectivity of the $e$-homomorphism, which are detailed in Section \ref{sec: proof_of_main_thm}. The first proof relies on looking at this problem from the perspective of Floer-Novikov theory developed by Lê-Ono \cite{Le-Ono} (see also \cite{Ono_FNCohomology} and \cite{Ono_FluxConjecture}). Floer-Novikov cohomology is a natural generalization of Hamiltonian Floer cohomology in the sense that to a symplectic path $\{\psi_{t}\}$ based at identity with non-degenerate endpoint $\psi=\psi_{1}$, we can associate a Floer-type cohomology group $\hfn^{*}(\{\psi_{t}\};J)$ that, up to a natural isomorphism, depends only on the flux of $\{\psi_{t}\}$. The proof has two key steps. First, using ideas in \cite{Ono_FluxConjecture} we obtain an isomorphism 
\begin{equation*}
	\hfn(\{\psi_{t}\};J)\cong\h^{*}(M;\Q)\otimes\Lambda_{\om},
\end{equation*}
for a symplectic path $\{\psi_{t}\}$ with $\wflux(\{\psi_{t}\})=[\eta]\in\ker e$. Next, we show that when $(M,\om)$ is spherically monotone the Floer-Novikov cohomology of a symplectic path $\{\psi_{t}\}$ is isomorphic to the Morse-Novikov cohomology $\hn^{*}(M, \eta)$  of its flux. A simple rank comparison then shows that this is only possible when $\wflux(\{\psi_{t}\})=0$, which concludes the argument. \\

\noindent The second proof follows from a result of McDuff \cite[Theorem 1]{McDuff_Monotone}, from which the triviality of $\Ga$ in the homologically monotone setting follows. Let $\{\psi_{t}\}$ be a symplectic loop. For a loop $\gamma$, set $\al(t,s)=\psi_{t}(\gamma(s))$ as before. McDuff's result implies that the $2$-cycle $A_{\gamma}$ represented by the torus $\im(\al)$ satisfies 
\begin{equation}\label{eq: mcduff}
	\brat{c_{1}(M),A_{\gamma}}=0.
\end{equation}
If $\{\psi_{t}\}$ has trivial evaluation, then $\im(\al)$ can be represented by a sphere. Thus, we obtain
\begin{equation*}
	\brat{\wflux(\{\psi_{t}\}),[\gamma]}=\brat{[\om],A_{\gamma}}=\lambda\cdot\brat{c_{1}(M),A_{\gamma}}=0,
\end{equation*}
where $\lambda\neq0$ is the monotonicity constant. Since $\gamma$ is arbitrary, we conclude that $\wflux(\{\psi_{t}\})=0$. While this proof is easier, it heavily relies on McDuff's result which was proven using highly non-trivial topological arguments. The first proof, on the other hand, is symplectic in nature.

\subsection{Symplectic circle actions}\label{sec: Symplectic_circle_actions}
Let $\rS^{1}=\R/\Z$ be the standard circle group. Let $(M,\om)$ be a symplectic manifold equipped with a smooth $\rS^{1}$-action  generated by a vector field $X$. The contraction of the symplectic form with the vector field $X$ defines a $1$-form $\iota_{X}\om$ on $M$. The circle action is called \textit{symplectic} whenever $\iota_{X}\om$ is closed, and \textit{Hamiltonian} if it is also exact. Knowing that the action is Hamiltonian has several advantages. For example, one can use a primitive $H$ of $\iota_{X}\om$, referred to as a \textit{moment map}, to obtain a symplectic quotient of $M$ at a regular value of $H$. This procedure is used, in particular, to reduce the dimension of the phase-space associated to problems arising in classical mechanics that have a circular symmetry. If $\h^{1}(M;\R)=0$, it is clear that every symplectic circle action is Hamiltonian. Otherwise, it becomes substantially more difficult to determine whether $\iota_{X}\om$ is an exact $1$-form. Finding sufficient and necessary conditions assuring that a symplectic circle action is Hamiltonian has been a subject of interest at least as early as the work of T. Frankel in the late 50s showing the following:
\begin{thm*}[T.Frankel \cite{Frankel}]
A symplectic circle action on a closed K\"{a}hler manifold is Hamiltonian if and only if it has fixed points.
\end{thm*}
\noindent This result was later generalized by Ono  \cite{Ono_Lefschetz} to closed Lefschetz manifolds. The condition of having fixed points was shown to be sufficient to guarantee the exactness of the circle action in a few other instances. McDuff \cite{McDuff_MomentMap}  proved it when $M$ has dimension four, while Ono  \cite{Ono_S1-ActionAspherical} and Ginzburg \cite{Ginzburg_RemarksS1Actions} showed it in the symplectically aspherical case. Tolman-Weitsman \cite{Tolman-Weitsman} proved that it remains true for semi-free circle actions with isolated fixed points. Finally, Lupton-Oprea \cite{Lupton-Oprea} and McDuff \cite{McDuff_Monotone} showed that every symplectic circle action is Hamiltonian when $(M,\om)$ is homologically monotone, i.e. the cohomology class $[\om]$ is a non-zero multiple of the first Chern class $c_{1}(M)$, a fact which is not true in the more general spherically monotone setting. Indeed, one can consider the symplectic circle action on $\rT^{2}\times\rS^{2}$ (with the product symplectic form of the standard symplectic structures) given by rotation in the first factor and identity on the second.  \\\

\noindent On the other hand,  McDuff constructed in \cite{McDuff_MomentMap} a non-Hamiltonian circle action with fixed tori on a closed $6$-dimensional Calabi-Yau symplectic manifold (see \cite{Cho-Kim_CY}), showing that the condition $M^{\rS^{1}}\neq\emptyset$ alone is not sufficient to guarantee the exactness of $\iota_{X}\om$ for general closed symplectic manifolds. McDuff and Salamon then asked if every symplectic circle action with isolated fixed points on a closed symplectic manifold is Hamiltonian. This question was answered in the negative by Tolman in \cite{Tolman_NonHamiltonianS1-Action}, see also \cite{Jang-Tolman}. It remains unclear when such examples can exist, or from another viewpoint, how large is the class of closed symplectic manifolds for which the presence of fixed points is equivalent to the exactness of the circle action. Nonetheless, the $e$-homomorphism gives a partial characterization of this class. In particular, we have the following:
\begin{prop}\label{prop: S1_actions}
Let $(M,\om)$ be a closed symplectic manifold such that the $e$-homomorphism is injective. Then, a symplectic circle action is Hamiltonian if and only if it has fixed points.
\end{prop}
\begin{proof}
If a symplectic circle action has a fixed point then the symplectic loop $\{\psi_{t}\}$ induced by it has trivial evaluation. The injectivity of the $e$-homomorphism implies that $[\iota_{X}\om]=\wflux(\{\psi_{t}\})=0$. Here, $X$ is the time-independent vector field generating the action.
\end{proof}
\noindent In view of Theorem \ref{thm: main}, we obtain the following:
\begin{thm}\label{thm: S1-action_monotone}
Let $(M,\om)$ be a closed spherically monotone symplectic manifold. Then, a symplectic circle action is Hamiltonian if and only if it has a fixed point.
\end{thm}

\noindent A result of McDuff \cite{McDuff_Uniruled} showed that every closed symplectic manifold that admits a Hamiltonian circle action is uniruled in the Gromov-Witten sense. In particular, they are geometrically uniruled: for each $\om$-compatible almost complex structure $J$ and each point $x\in M$, there is a non-constant $J$-holomorphic sphere $u$ such that $x\in\im(u)$. Symplectic manifolds that are symplectically Calabi-Yau are not geometrically uniruled, neither are (spherically) negative monotone ones. By considering this fact in addition to Theorem \ref{thm: S1-action_monotone} we obtain the following corollary:
\begin{cor}\label{thm: fixed_pts_of_S1-aciton}
Let the circle act symplectically and non-trivially on a closed symplectic manifold $(M,\om)$ such that \[c_{1}(M)|_{\pi_{2}(M)}=\la\cdot[\om]|_{\pi_{2}(M)},\] for $\la\in\R$. Then
\begin{itemize}
    \item[(i)] If $\la>0$, the action is Hamiltonian if and only if it has fixed points.
    \item[(ii)] If $\la<0$, the action is non-Hamiltonian, and has no fixed points.
    \item[(iii)] If $\la=0$ the action is non-Hamiltonian.
\end{itemize}
\end{cor}
\begin{remark}\label{rmk: T2xM} 
When $\la\neq0$ there are examples of non-Hamiltonian symplectic circle actions. In particular, consider the symplectic product $\mathbb{T}^2\times M$ of the standard symplectic torus with any closed spherically monotone symplectic manifold, and the symplectic circle action given by $t\cdot(x,y,p)=(x+t,y,p)$ for $(x,y)\in T^{2}$ and $p\in M$.
\end{remark}
\noindent Another closely related question was raised by McDuff-Salamon \cite{McDuffSalamonIntro3}. They asked if there exists a symplectic free circle action whose orbits are contractible. Kotschick \cite{Kotschick_FreeS1-Actions} proved this to be the case for all symplectic manifolds of dimension four, even if the action is only assumed to be smooth. Furthermore, Kotschick produced examples of symplectic free circle actions with contractible orbits in every even dimension greater or equal to six. As corollary of Theorem \ref{thm: main} and the argument in the proof of Proposition \ref{prop: S1_actions} we obtain the following:
\begin{thm}\label{thm: free_S1-action}
Let $(M,\om)$ be a closed symplectic manifold satisfying $(\bigstar)$. Then, every free symplectic circle action must have non-contractible orbits. 
\end{thm}

\subsection{Applications to symplectic torsion}
The injectivity of the $e$-homorphism, when combined with results in \cite{AtallahShelukhin_Torsion1}, provides applications to questions about the existence of finite subgroups of $\Symp_{0}(M,\om)$. In \cite{AtallahShelukhin_Torsion1}, it was shown that if $(M,\om)$ is positive spherically monotone, then the existence of a non-trivial finite subgroup of $\Ham(M,\om)$ implies that $(M,\om)$ is geometrically uniruled. Furthermore, it was shown that if $(M,\om)$ is negative spherically monotone, then there are no non-trivial finite subgroups of $\Ham(M,\om)$. Therefore, as a corollary of Propostion \ref{prop: vanishing_of_flux_group} and Theorem \ref{thm: main} we obtain the following:
\begin{thm}\label{thm: Symp_no-torsion}\label{thm: symptorsion1}
Let $(M,\om)$ be a closed symplectic manifold such that \[c_{1}(M)|_{\pi_{2}(M)}=\la\cdot[\om]|_{\pi_{2}(M)},\] for $\la\in\R\setminus\{0\}$. Further suppose that either $\chi(M)\neq0$ or $\pi_{1}(M)$ has finite center. Then
\begin{itemize}
    \item[(i)] If $\la>0$, then the existence of a non-trivial finite subgroup of $\Symp_0(M,\om)$ implies that $(M,\om)$ is geometrically uniruled.
    \item[(ii)] If $\la<0$, then all finite subgroups of $\Symp_0(M,\om)$ are trivial.
\end{itemize}
\end{thm}
\begin{remark}
The symplectic product of $(\C P^{n}, \om_{\mrm{FS}})$ with any closed oriented surface $\Sigma_{g}$ of genus $g\geq2$ with the standard area form satisfies condition (i) of Theorem \ref{thm: symptorsion1}. More generally, any symplectic product of the form $(M\times N, \om_{M}\oplus\om_{N})$, where $(M,\om_{M})$ is positive spherically monotone with $\chi(M)\neq0$ and $(N,\om_{N})$ is symplectically aspherical with $\chi(N)\neq0$. Furthermore, following \cite{dimca2012singularities}, the Euler characteristic of the degree $m$ hypersurface $X_{m}\subset\C P^{n}$ defined by setting $z_{0}^{m} + \cdots + z_{n}^{m}=0$ is given by \[\chi(X_{m}) = \frac{1}{m}\big((1-m)^{n+1} - 1\big) + n + 1.\] Therefore, $\chi(X_{m})\neq0$ when $m>n+1$. Hence, $X_{m}$, which is negative monotone, satisfies condition (ii) of Theorem \ref{thm: symptorsion1}.
\end{remark}

\noindent The injectivity of the $e$-homomorphism also gives information about the presence of fixed points of a non-trivial symplectic diffeomorphism $\psi\in\Symp_{0}(M,\om)$ of finite order. The following definition given by Polterovich in \cite{Polterovich_Distortion} naturally fits into this context.
\begin{definition}\label{df: contractible_type}
A fixed point $x$ of a symplectic diffeomorphism $\psi\in\Symp_{0}(M,\om)$ is of \textit{contractible type} if there exists a symplecitc path $\{\psi_{t}\}$ based at the identity with $\psi_{1}=\psi$ such that the loop $\{\psi_{t}(x)\}$ is contractible in $M$.
\end{definition}
\noindent The presence of a fixed point of contractible type of a non-trivial $\psi\in\Symp_{0}(M,\om)$ of finite order implies that it must be Hamiltonian. Indeed, we have the following:
\break
\begin{prop}\label{prop: torsion}
Let $(M,\om)$ be a closed symplectic manifold such that the e-homomorphism is injective. Further suppose that $\psi\in\Symp_{0}(M,\om)$ is non-trivial of finite order, i.e. $\psi^d=\id$ for some integer $d>1$. Then $\psi$ is Hamiltonian if and only if it admits a fixed point of contractible type.
\end{prop}
\begin{proof}
Suppose $x\in\fix(x)$ is of contractible type. Let $\{\psi_{t}\}$ be a symplectic path with $\psi_{1}=\psi$ such that $\{\psi_{t}(x)\}$ is contractible in $M$. Note that $\{\psi^{d}_{t}\}$ is a symplectic loop. Then,
	\begin{equation*}
	e(\wflux(\{\psi^d_{t}\}))=[\{\psi^{d}_{t}(x)\}]=[\{\psi_{t}(x)\}]^{d}=1.
	\end{equation*}
Hence, by the injectivity of $e$ and the fact that $\wflux$ is a homomorphism to a torsion free group, we have that $\wflux(\{\psi_{t}\})=0$. In particular, $\psi$ is Hamiltonian.
\end{proof}
\begin{cor}
Let $(M,\om)$ be a closed symplectic manifold satisfying $(\bigstar)$ such that $\chi(M)\neq0$. Then, all non-trivial symplectic diffeomorphisms $\psi\in\Symp_{0}(M,\om)$ of finite order are Hamiltonian. 
\end{cor}
\noindent When the Euler characteristic of $(M,\om)$ is zero and $\psi\in\Symp_{0}(M,\om)$ is torsion, Propostion \ref{prop: appendix} tells us that it either has no fixed points or non-isolated ones. In the latter case, there is no a priori reason for the existence of a fixed point of contractible type. However, in the case of the standard symplectic torus, Polterovich \cite{Polterovich_Distortion} showed that any fixed point of a symplectic diffeomorphism $\psi\in\Symp_{0}(\rT^{2n}dp\wedge dq)$ is of contractible type. We can adapt that argument to prove the following:
\begin{cor}\label{cor: Torsion_on_T2}
Let $(M,\om)$ be a simply connected closed symplectic manifold satisfying $(\bigstar)$. Then, every non-Hamiltonian $\psi\in\Symp_{0}(\rT^{2n}\times M, dp\wedge dq\oplus\om)$ of finite order has no fixed points.
\end{cor}
\begin{proof}
Let $(x,p)\in\rT^{2n}\times M$ be a fixed point of $\psi$ and $\{\psi_{t}\}$ a symplectic path such that $\psi_{0}=\id$ and $\psi_{1}=\psi$. Consider the lift \[\til{\psi}_{t}:\R^{2n}\times M\ra\R^{2n}\times M\] of $\psi_{t}$ to the universal cover of $\rT^{2n}\times M$ and pick $(\til{x},p)\in\pi^{-1}(\{(x,p)\})$. Then, $\til{\psi}(\til{x},p)=(\til{x} + a,p)$ for some $a\in\Z^{2n}$. We can then define a symplectic flow $f_{t}\times\id_{M}$ on $\rT^{2n}\times M$ by setting \[f_{t}(y)=y-t\cdot a\quad\mrm{mod}\quad1.\] Note that its lift to the universal cover is given by \[(\til{f}_{t}\times\id_{M})(y,q)=(y-t\cdot a, q).\] Therefore, by setting $\varphi_{t}=f_{t}\circ\psi_{t}$, we obtain a symplectic path based at the identity with $\varphi_{1}=\psi$, and whose lift $\til{\varphi}_{t}$ satisfies $\til{\varphi}(\til{x},p)=(\til{x},p)$. Consequently, the loop $\{\varphi_{t}(x,p)\}$ is contractible in $\rT^{2n}\times M$. The claim of the corollary then follows by noting that $T^{2n}\times M$ satisfies condition $(\bigstar)$ and by Proposition \ref{prop: torsion}.
\end{proof}

\section{Preliminaries}
\subsection{Floer-Novikov Cohomology}\label{sec: FNC}
In Section \ref{sec: FNH_Flux_Conjecture} we review the construction of Floer-Novikov cohomology after Ono \cite{Ono_FluxConjecture}. This is a cohomological version of the construction introduced in Lê-Ono \cite{Le-Ono} with a slightly smaller coefficient ring. In Section \ref{sec: FNH_Variant} we review a variant of Floer-Novikov cohomology introduced in \cite{Ono_FNCohomology} which enables the comparison between symplectic paths with different flux. We refer to \cite{Le-Ono_Ukraine} for an in-depth discussion of the variants of Floer-Novikov cohomology and the relations among them. We also briefly recall the definition of classical Morse-Novikov cohomology and outline a few important properties it satisfies. For further details on Morse-Novikov cohomology see \cite{Le-Ono}, \cite{PolterovichRosen}, and \cite{FarberBook}. 

\subsubsection{Morse-Novikov Cohomology}
Let $M$ be a closed smooth manifold and let $\theta$ be a closed $1$-form on $M$.  Fix a ground field $\bK$. For our purposes it will be sufficient to consider the case $\bK=\Q$. Denote by $I_{\theta}: \pi_{1}(M)\ra\R$ the homomorphism given by integrating $\theta$ over a representative loop, that is:
\begin{equation*}
	I_{\theta}([\ga])=\int_{\ga}\theta. 
\end{equation*}
In particular, $I_{\theta}$ depends only on the cohomology class $[\theta]$. Denote by $\pi:\ol{M}^{\theta}\ra M$ the covering space of $M$ corresponding to $\ker I_{\theta}\subset\pi_{1}(M)$. It is the minimal abelian covering on which $\pi^{*}\theta$ is exact. The covering transformation group is given by $G_{\theta}=\pi_{1}(M)/\ker I_{\theta}$. We denote by $\Lambda_{\theta}$ the completion of the group ring of $G_{\theta}$ with respect to the filtration induced by $I_{\theta}$, that is:
\begin{equation*}
	\Lambda_{\theta}=\bigg\{\sum_{i}a_{i}g_{i}\,\Big|\,a_{i}\in\bK,\,g_{i}\in G_{\theta},\,\text{satisfying (A)}\Big\},
\end{equation*}
\begin{center}
	(A) For each $c\in\R$ the set $\{i\,|\,a_{i}\neq0, I_{\theta}(g_{i})<c\}$ is finite.
\end{center}
\vspace{0.2cm}
The fact that $\bK$ is a field implies that so is $\Lambda_{\theta}$. Let $\ol{f}$ be a choice of primitive for $\pi^{*}\theta$. Then, to each $\til{x}\in\crit(\ol{f})$ there corresponds a zero $x=\pi(\til{x})$ of $\theta$. A $1$-form $\theta$ is said to be Morse if all the critical points of $\ol{f}$ are non-degenerate. The Morse-Novikov cochain complex in degree $k$ with coefficients in $\bK$ is defined as:
\begin{equation*}
	\cn^{k}(M,\theta)=\bigg\{\sum_{i}a_{i}\til{x}_{i}\,\Big|\,a_{i}\in\bK,\,\til{x}_{i}\in\crit(\ol{f}),\,\mrm{index}(\til{x}_{i})=k, \,\text{satisfying (B)}\Big\},
\end{equation*}
\begin{center}
	(B) For each $c\in\R$ the set $\{i\,|\,a_{i}\neq0,\ol{f}(\til{x}_{i})<c\}$ is finite.
\end{center}
\vspace{0.2cm}
Note that $\cn^{*}(M,\theta)$ is finitely generated over $\Lambda_{\theta}$. For a choice of Riemannian metric $g$ on $M$, the coboundary operator $\delta$ is defined by counting bounded gradient trajectories of $\ol{f}$ with respect to the pullback metric $\pi^{*}g$ that are emerging from a critical point $\til{x}\in\crit(\ol{f})$ and converging to critical points $\til{y}\in\crit(\ol{f})$ such that $\mrm{index}(\til{y})-\mrm{index}(\til{x})=1$. We may assume that the gradient flow is of Morse-Smale type. The Morse-Novikov cohomology of $\theta$ is defined as $\hn^{*}(M,\theta)=H^{*}((\cn(M,\theta),\delta))$, and is a finitely generated $\Lambda_{\theta}$-module. The resulting cohomology is independent of the choice of Riemannian metric for which the flow is of Morse-Smale type. Furthermore, cohomologous $1$-forms have canonically isomorphic Morse-Novikov cohomologies. That is, if $[\theta_{1}]=[\theta_{2}]$, then $\hn^{*}(M,\theta_{1})$ is isomorphic to $\hn^{*}(M,\theta_{2})$ as graded vector spaces over $\Lambda_{\theta_{1}}=\Lambda_{\theta_{2}}$. More generally, we have the following:
\begin{prop}[{Lê-Ono, \cite[Theorem C.2]{Le-Ono}}]\label{prop: Le-Ono_C.2}
Let $\theta_{1}$ and $\theta_{2}$ be closed $1$-forms such that $\ker I_{\theta_{1}}\subset\ker I_{\theta_{2}}$. Then, for each degree $k$, \[\dim_{\Lambda_{\theta_{1}}}\hn^{k}(M,\theta_{1})\leq\dim_{\Lambda_{\theta_{2}}}\hn^{k}(M,\theta_{2}).\]
\end{prop}
\noindent In addition, we have the following useful proposition, which distinguishes Morse-Novikov cohomology from usual Morse cohomology.
\begin{prop}[{Ono, \cite[Proposition 4.12]{Ono_FluxConjecture}}]\label{prop: Morse_v_MorseNovikov}
Let $\theta$ be a closed $1$-form such that $[\theta]\neq0$. Then \[\dim_{\Lambda_{\theta}}\hn^{k}(M,\theta)=0\] for $k=0$ or $k=\dim(M)$.		
\end{prop}
\noindent In particular, when $\theta$ is not exact, we have that 
\begin{equation}
\dim_{\Lambda_{\theta}}\hn(M,\theta)<\dim_{\Q}\mrm{H}(M).
\end{equation}

\subsubsection{Floer-Novikov cohomology}\label{sec: FNH_Flux_Conjecture}
Let $(M,\om)$ be a closed symplectic manifold, and $\{\psi_{t}\}$ a symplectic path based at the identity with endpoint $\psi_{1}=\psi$. The Deformation lemma in \cite{Le-Ono} implies that we can suppose, without loss of generality, that there exists a $1$-periodic smooth family of smooth functions $F_{t}\in{C}^{\infty}(M)$, $F_{t+1}=F_{t}$, such that
\begin{equation}\label{eq: deformation_lma}
\iota_{X^{t}}\om=\theta+dF_{t},
\end{equation}
where $X^{t}$ is the time-dependent vector field induced by the symplectic isotopy $\{\psi_{t}\}$ and $\theta$ is a closed $1$-form representing $\wflux(\{\psi_{t}\})$. When $\theta$ is exact Equation (\ref{eq: deformation_lma}) becomes the usual Hamiltonian equation for $H_{t}=F_{t}+f$, where $df=\theta$. Similarly, we have a formal closed $1$-form on the contractible component $\cL M$ of the loop space of $M$. Indeed, for a loop $x\in\cL M$ and $v\in T_{x}\cL M=\Gamma(x^{*}TM)$, we define
\begin{align*}
a_{\{\psi_{t}\}}(v)&=\int^{1}_{0}\om(v(t),x'(t)-X^{t}(x(t)))dt\\
	&=\int^{1}_{0}\om(v(t),x'(t))dt + \int^{1}_{0}(\theta + dF_{t})(v(t))dt.
\end{align*}
The idea is to find a suitable cover on which $a_{\{\psi_{t}\}}$ is exact, and define Floer-Novikov cohomology as an analog of Morse theory for a primitive $\cA_{\psi_{t}}$. Consider the following homomorphisms $\cI_{\theta},\cI_{\om}, \cI_{c_{1}}:\pi_{1}(\cL M)\ra\R$ defined as 
\begin{equation*}
\cI_{\theta}=I_{\theta}\circ \mrm{Ev}_{*},\quad \cI_{\om}(\{x_{s}\})=\int_{C(\{x_{s}\})}\om,\quad\cI_{c_{1}}(\{x_{s}\})=\brat{c_{1}(M),[C(\{x_{s}\})]}
\end{equation*}
for a loop $\{x_s\}_{s\in[0,1]}$ in $\cL M$. Here, $\mrm{Ev}:\cL M\ra M$ is given by evaluation at $t=0$, and $C(\{x_{s}\})$ is the $2$-cycle represented by \[(s,t)\in\rS^{1}\times\rS^{1}\mapsto x_{s}(t)\in M.\] We denote by $\til\cL M\ra\cL M$ the covering space of $\cL M$ corresponding to \[\ker(\cI_{\om}+\cI_{\theta})\cap\ker\cI_{c_{1}}\subset\pi_{1}(\cL M).\] A description of $\til\cL M$ can be given in the following manner. Choose a primitive $\ol{f}$ of $\pi^{*}\theta$ on $\ol{M}^{\theta}$, and consider pairs $(\til{x},u)$ composed of a loop $\til{x}\in\cL\ol{M}^{\theta}$ and a capping $u:\D\ra M$ such that $u|_{\del\D}=\pi\circ\til{x}$. Define the equivalence relation given by $(\til{x},u)\sim(\til{y},w)$ if and only if
 \begin{align*}
 	\pi\circ\til{x}&=\pi\circ\til{y}\\
	\int_{u}\om+\ol{f}(\til{x}(0))&=\int_{w}\om+\ol{f}(\til{y}(0)),
 \end{align*}
 and
 \begin{align*}
 	\brat{c_{1}(M), u\#(-w)}&=0.
 \end{align*}
Here, $u\#(-w)$ corresponds to the sphere obtained by gluing the two disks along their common boundaries with the orientation of $w$ reversed. Each element in $\til\cL M$ corresponds to an equivalence class $[\til{x},u]$ of such a pair. With this description, the covering map $\Pi:\til\cL M\ra\cL M$ is given by $[\til{x}, u]\mapsto \pi\circ\til{x}=x$. A choice of primitive for $\Pi^{*}a_{\{\psi_{t}\}}$ is given by
\begin{equation*}
\cA_{\{\psi_{t}\}}([\til{x},u])= \int^{1}_{0}(\ol{f}+F_{t}\circ\pi)(\til{x}(t))dt + \int_{u}\om,
\end{equation*}
where $F_{t}$ is as in Equation (\ref{eq: deformation_lma}). The critical points $\cP(\{\psi_{t}\})$ of $\cA_{\{\psi_{t}\}}$ are the lifts to $\til\cL M$ of the fixed points $x\in\fix(\psi)$ of $\psi$ such that $[\{\psi_{t}(x)\}]=1$ in $\pi_{1}(M)$. To a critical point $[\til{x},u]$ we assign an index $\cz([\til{x},u])$ given by the Conley-Zehnder index of $x=\pi\circ\til{x}$ with respect to the trivialization $x^{*}TM$, which extends to $u^{*}TM$. Note that the covering transformation group of $\cL M$ is given by
\begin{equation*}
	G_{\theta,\om}= \frac{\pi_{1}(\cL M)}{\ker(\cI_{\om}+\cI_{\theta})\cap\ker\cI_{c_{1}}}.
\end{equation*}
Let $\Lambda_{\theta,\om}$ be Novikov ring given by the completion of the group ring of $G_{\theta,\om}$ with respect to the filtration induced by $\cI_{\om}+\cI_{\theta}$, that is:
\begin{equation*}
	\Lambda_{\theta, \om}=\bigg\{\sum_{i}a_{i}g_{i}\,\Big|\,a_{i}\in\Q,\,g_{i}\in G_{\theta, \om},\,\text{satisfying (A')}\Big\}.
\end{equation*}
\begin{center}
	(A') For each $c\in\R$ the set $\{i\,|\,a_{i}\neq0, (\cI_{\om}+\cI_{\theta})(g_{i})<c\}$ is finite.
\end{center}
\vspace{0.2cm}
Suppose that $\psi$ is non-degenerate and let $J=\{J_{t}\}$ be a family of $\om$-compatible almost complex structures. The Floer-Novikov cochain complex is defined as follows:
\begin{equation*}
	\cfn^{k}(\{\psi_{t}\};J)=\bigg\{\sum_{i}a_{i}[\til{x}_{i},u_{i}]\,\Big|\,a_{i}\in\Q,\,[\til{x}_{i},u_{i}]\in\cP(\{\psi_{t}\}),\,\text{satisfying (B')}\bigg\}
\end{equation*}
\begin{center}
	(B') For each $c\in\R$ the set $\{i\,|\,a_{i}\neq0,\cA_{\{\psi_{t}\}}([\til{x},u])<c\}$ is finite, and, for all $i$, $\cz([\til{x}_{i},u_{i}])=k$.
\end{center}
\vspace{0.2cm}
\noindent The graded $\Q$-vector space $\cfn^{*}(\{\psi_{t}\};J)$ is endowed with the Floer-Novikov coboundary operator $\delta_{\mrm{FN}}$, which is defined as the signed count of isolated solutions (modulo the $\R$-action) of the asymptotic boundary value problem on maps $u:\R\times\rS^{1}\ra M$ defined by the gradient of $\cA_{\{\psi_{t}\}}$ \cite{Le-Ono, Ono_FluxConjecture}. In other words, the coboundary operator counts the finite energy solutions to the Floer equation
\begin{equation*}
	\frac{\del u}{\del s}+J_{t}(u)\bigg(\frac{\del u}{\del t} - X_{t}(u)\bigg) = 0, \qquad \lim_{s\ra\pm\infty}\til{u}(s,t)=\til{x}^{\pm}(t),
\end{equation*}
for some lift $\til{u}:\R\times\rS^{1}\ra M$, such that $\cz([\til{x}^{+},u^{+}])-\cz([\til{x}^{-},u^{-}])=1$, and $[\til{x}^{+}, u^{-}\#u]=[\til{x}^{+},u^{+}]$. The Floer-Novikov cohomology is defined as \[\hfn^{*}(\{\psi_{t}\};J)= H^{*}((\cfn(\{\psi_{t}\};J),\delta_{\mrm{FN}}).\] It can be shown to be independent of the choice of $J$. In addition, it only depends on the flux of the symplecitc path.
\pagebreak
\begin{thm}[{Lê-Ono, \cite[Theorem 4.3]{Le-Ono}}]
Let $\{\psi^{(1)}_{t}\}$ and $\{\psi^{(2)}_{t}\}$ be symplectic paths with non-degenerate endpoints. Suppose that $\{\psi^{(1)}_{t}\circ(\psi^{(2)}_{t})^{-1}\}$ has zero flux.. Then,
\begin{equation*}
	\hfn^{*}(\{\psi^{(1)}_{t}\}; J^{(1)})\cong\hfn^{*}(\{\psi^{(2)}_{t}\}; J^{(2)})
\end{equation*}
as $\Lambda_{\theta,\om}$-modules. Here, $[\theta]$ is the flux of both paths. 
\end{thm}

\begin{remark}\label{rmk: FN_to_F}	
When $\{\varphi_{H}^{t}\}$ is a Hamiltonian isotopy, $\ol{M}^{\theta}=M$ and $(\cfn^{*}(\{\psi\};J),\delta)$ reduces to the usual Floer cochain complex $(\cf(H;J),\delta)$ associated with the Hamiltonian function $H$.  Therefore, the PSS isomorphism yields 
\begin{equation}
\hfn^{*}(\{\varphi_{H}^{t}\};J)=\hf^{*}(H;J)\cong\mrm{H}^{*+n}(M;\Q)\otimes_{\Q}\Lambda_{\om}.
\end{equation}
\end{remark}
\noindent The main reason for the choice of Novikov ring in this construction is so that we have the following two statements.
\begin{thm}[{Ono, \cite[Theorem 4.10]{Ono_FluxConjecture}}]\label{thm: FNH_of_loop}
Let $\{\psi^{(1)}_{t}\}$ and $\{\psi^{(2)}_{t}\}$ be symplectic paths with non-degenerate endpoints $\psi^{(1)}_{1}=\psi^{(2)}_{1}=\psi$. Suppose that the symplectic loop $\{\psi^{(1)}_{t}\circ(\psi^{(2)}_{t})^{-1}\}$ has trivial evaluation, i.e. $ev(\{\psi^{(1)}_{t}\circ(\psi^{(2)}_{t})^{-1}\})=1$.  Then, we have a ring isomorphism $\Psi:\Lambda_{\theta_{1},\om}\xrightarrow{\cong}\Lambda_{\theta_{2},\om}$, and 
\begin{equation*}
	\hfn^{*}(\{\psi^{(1)}_{t}\}; J^{(1)})\cong\hfn^{*}(\{\psi^{(2)}_{t}\}; J^{(2)})
\end{equation*}
as $\Lambda_{\theta_{1},\om}$-modules, where $[\theta_{i}]$ corresponds to the flux of $\{\psi_{t}^{(i)}\}$ for $i=1,2$. The module action of $\Lambda_{\theta_{1},\om}$ on $\hfn^{*}(\{\psi^{(2)}_{t}\}, J^{(2)})$ is the one induced by $\Psi$. 
\end{thm}
\begin{thm}[{Ono, \cite[Theorem 3.12]{Ono_FluxConjecture}}]\label{thm: FNH_and_MNH}
Let $\{\psi_{t}\}$ be a symplectic path based at identity with sufficiently small flux $[\eta]$. Then,
	\begin{equation*}
		\hfn^{*}(\{\psi_{t}\}; J)\cong\hn^{*+n}(M, \eta)\otimes_{\Lambda_{\eta}}\Lambda_{\eta,\om}.
	\end{equation*}
\end{thm}

\begin{remark}\label{rmk: no_PSS}
The Floer-Novikov cohomology $\hfn^{*}(\{\psi_{t}\}; J)$ of a symplectic path $\{\psi_{t}\}$ is not always isomorphic to the Morse-Novikov cohomology $\hn(M,\theta)$ of its flux $[\theta]$ (this has also been observed in \cite{Seidel_Braids}). This can be seen by studying the examples in \cite{Jang-Tolman, Tolman_NonHamiltonianS1-Action, McDuff_MomentMap} of non-Hamiltonian symplectic circle actions with fixed points. Indeed, in these cases we have symplectic paths with non-trivial flux and trivial evaluation. In particular, if we suppose that such an isomorphism exists, Theorem \ref{thm: FNH_of_loop} together with Theorem \ref{thm: FNH_and_MNH} would imply that $\dim_{\Lambda_{\theta}}\hn(M,\theta) = \dim_{\Q}\h(M;\Q)$, which is in contradiction to Proposition \ref{prop: Morse_v_MorseNovikov}. This also shows that the $e$-homomorphism is not injective in general.
\end{remark}

\subsubsection{A variant of Floer-Novikov cohomology for changing flux}\label{sec: FNH_Variant}
In this section we recall a variant, introduced in \cite{Ono_FNCohomology}, of the Floer-Novikov cohomology construction presented in Section \ref{sec: FNH_Flux_Conjecture}, which allows the comparison between symplectic paths that have different flux.\\

\noindent Let $(M,\om)$ be a closed symplectic manifold. Suppose $\{\psi_{t}\}$ is a symplectic path with endpoint $\psi_{1}=\psi$ and flux $[\theta]$. Let $p:\til{M}\ra M$ be an abelian covering space of $M$ such that $p^{*}\theta$ is exact; $\ol{M}^{\theta}$ is the smallest choice of such a covering space. Let $\{\til{F}_{t}\}$ be a smooth family of smooth functions on $\til{M}$ such that $d\til{F}_{t}=p^{*}\theta$. Just as before, we would like to make a choice of covering space of $\cL M$ on which the pullback of $a_{\{\psi_{t}\}}$ is exact. We denote by $P:\til\cL\til{M}\ra\cL M$ the covering space of $\cL M$ associated with
\begin{equation*}
	\ker\cI_{\om}\cap\ker\cI_{c_1}\cap \mrm{Ev}_{*}^{-1}(p_{*}(\pi_{1}(\til{M})))\subset\pi_{1}(\cL M).
\end{equation*}
This covering can be seen as the space of pairs $(\til{x},u)$, $u|_{\del\D}=\pi\circ\til{x}$, under the following equivalence relation: $(\til{x},u)\sim(\til{y},w)$ if and only if
 \begin{align*}
 	\til{x}&=\til{y}\\
	\int_{u}\om&=\int_{w}\om,
 \end{align*}
 and
 \begin{align*}
 	\brat{c_{1}(M), u\#(-w)}&=0.
 \end{align*}
On $\til\cL M$, the pullback $P^{*}a_{\psi_{t}}$ is exact and a choice of primitive is given by
\begin{equation*}
\til{\cA}_{\{\psi_{t}\}}([\til{x},u])= \int^{1}_{0}\til{F}_{t}(\til{x}(t))dt + \int_{u}\om.
\end{equation*}
Similairly, the critical points $\til\cP(\{\psi_{t}\})$ of $\til\cA_{\{\psi_{t}\}}$ are lifts to $\til\cL\til{M}$ of the fixed points $x\in\fix(\psi)$ of $\psi$ that satisfy $[\{\psi_{t}(x)\}]=1\in\pi_{1}(M)$. Just as in Seciton \ref{sec: FNH_Flux_Conjecture}, to each critical point $[\til{x},u]$ we assign a Conley-Zehnder type index. The covering transformation group of $\til\cL\til{M}$ is given by
\begin{equation*}
	\til{G}_{\theta,\om}= \frac{\pi_{1}(\cL M)}{\ker\cI_{\om}\cap\ker\cI_{c_1}\cap \mrm{Ev}_{*}^{-1}(p_{*}(\pi_{1}(\til{M})))}\,,
\end{equation*}
and, we denote by $\til\Lambda_{\theta,\om}$ the Novikov ring given by the completion of its group ring with respect to the filtration induced by $\cI_{\om}+\cI_{\theta}$, that is:
\begin{equation*}
	\til\Lambda_{\theta, \om}=\bigg\{\sum_{i}a_{i}g_{i}\,\Big|\,a_{i}\in\Q,\,g_{i}\in \til{G}_{\theta, \om},\,\text{satisfying (A')}\Big\}.
\end{equation*}
Suppose that $\psi$ is non-degenerate and let $J=\{J_{t}\}$ be a family of $\om$-compatible almost complex structures. The Floer-Novikov cochain complex is defined as follows:
\begin{equation*}
	\cfn^{k}(\{\psi_{t}\},\til{M};J)=\bigg\{\sum_{i}a_{i}[\til{x}_{i},u_{i}]\,\Big|\,a_{i}\in\Q,\,[\til{x}_{i},u_{i}]\in\til\cP(\{\psi_{t}\}),\,\text{satisfying (B')}\bigg\}.
\end{equation*}
The coboundary operator $\til{\delta}_{FN}$ is defined by the same formula as in Section \ref{sec: FNH_Flux_Conjecture}. The Floer-Novikov homology $\hfn_{*}(\{\psi_{t}\}, \til{M};J)$ associated with the covering space $\til{M}$ is defined as the homology of $(\cfn(\{\psi_{t}\},\til{M};J),\til{\delta}_{FN})$. It is independent of the choice of almost complex structure $J$ and depends only on the flux of the symplecitc path. The following theorem allows the comparison between the ranks of the Floer-Novikov cohomology of symplectic paths with flux lying in the kernel of $p^{*}$. 

\begin{prop}[{Ono, \cite[Proposition 4.8]{Ono_FNCohomology}}]\label{thm: FNH1_to_FNH2}
Let $\{\psi^{(1)}_{t}\}$ and $\{\psi^{(2)}_{t}\}$ be symplectic paths such that $\wflux(\{\psi^{(i)}_{t}\})\in\ker\{p^{*}:\h^{1}(M;\R)\ra\h^{1}(\til{M};\R)\}$ for $i=1,2$. Then,
\begin{equation*}
	\rank_{\til\Lambda_{\theta_{1}, \om}}\hfn^{*}(\{\psi^{(1)}_{t}\},\til{M}; J^{(1)})=\rank_{\til\Lambda_{\theta_{2}, \om}}\hfn^{*}(\{\psi^{(2)}_{t}\},\til{M}; J^{(2)}),
\end{equation*}
where $[\theta_{i}]$ corresponds to the flux of $\{\psi_{t}^{(i)}\}$ for $i=1,2$.
\end{prop}
\begin{remark}\label{rmk: HFN_variants}
In general, $\hfn_{*}(\{\psi_{t}\}, \til{M};J)$ is different from $\hfn_{*}(\{\psi_{t}\} ;J)$ even when $\til{M}=\ol{M}^{\theta}$. Nonetheless, whenever $\ker\cI_{c_{1}}\subset\ker\cI_{\om}$ we have that \[\hfn_{*}(\{\psi_{t}\}, \ol{M}^{\theta};J)=\hfn_{*}(\{\psi_{t}\} ;J).\] This holds for $(M,\om)$ spherically (positive or negative) monotone or symplectically aspherical. We refer to \cite{Le-Ono_Ukraine} for more details on the relationship between these cohomology theories. 
\end{remark}

\section{Proof of Theorem \ref{thm: main}}\label{sec: proof_of_main_thm}
\subsection{Proof using Floer-Novikov theory}
When $(M,\om)$ is either symplectically aspherical or satisfies the weak Lefschetz property, we have presented proofs based on classical arguments in Section \ref{sec: e-homomorphism}. We shall, therefore, consider the spherically monotone case. Let $\{\psi_{t}\}$ be a symplectic loop with \[\wflux(\{\psi_{t}\})=[\eta]\in\Gamma.\] Observe that spherical monotonicity implies that  
\begin{equation*}\label{eq: equality_of_rings}
	\ker(\cI_{\om}+\cI_{\eta})\cap\ker\cI_{c_{1}} = \ker\cI_{\om}\cap\ker\cI_{\eta}\cap\ker\cI_{c_{1}}.
\end{equation*}
Set $\til{M}=\ol{M}^{\eta}$ in the construction of the variant of Floer-Novikov homology defined in Section \ref{sec: FNH_Variant}. In this case, note that $\mrm{Ev}_{*}^{-1}(p_{*}(\pi_{1}(\ol{M}^{\eta}))=\ker\cI_{\eta}$. Indeed, this is a consequence of $\pi_{1}(\ol{M}^{\eta})=\ker I_{\eta}$, which follows from the defining property of $\ol{M}^{\eta}$, and of the equality $\cI_{\eta}=I_{\eta}\circ \mrm{Ev}_{*}$. These observations allow to deduce that for any symplectic path $\{\xi_{t}\}$ with flux $[\eta]$, we have $\til{\Lambda}_{\eta, \om}={\Lambda}_{\eta, \om}$, and
\begin{equation}\label{eq: equality_of_complexes}
	(\cfn^{*}(\{\xi_{t}\},\ol{M}^{\eta};J), \til\delta_{FN})=(\cfn^{*}(\{\xi_{t}\};J), \delta_{FN})
\end{equation}
by definition (we reiterate that monotonicity is used in an important way, see Remark \ref{rmk: HFN_variants}). Now, further suppose that $[\{\psi_{t}\}]\in\ker ev$. Let $\{\varphi_{t}\}$ be a Hamiltonian path with endpoint $\varphi$, which is generated by a non-degenerate Hamiltonian $H$, and set $\psi'_{t}=\psi_{t}\circ\varphi_{t}$. Then, $\{\psi'_{t}\}$ and $\{\varphi_{t}\}$ are two symplectic paths with non-degenerate endpoints $\psi'=\varphi$. Since $ev(\{\psi_{t}\}) = 1$, Theorem \ref{thm: FNH_of_loop} implies that $\Lambda_{\eta,\om}\cong\Lambda_{\om}$ and that
\begin{equation}\label{eq: main1}
		\hfn^{*}(\{\psi'_{t}\}, J')\cong\hfn^{*}(\{\varphi_{t}\}; J)
\end{equation}
as $\Lambda_{\eta,\om}$-modules. Since $\varphi_{t}$ is Hamiltonian, we have that
\begin{equation}\label{eq: main2}
	\hfn^{*}(\{\varphi_{t}\}; J)\cong\hf^{*}(H; J)\cong \h^{*}(M;\Q)\otimes_{\Q}\Lambda_{\om}.
\end{equation}
For $\varepsilon>0$, let $\{\psi^{\varepsilon\eta}_{t}\}_{t\in[0,1]}$ be the symplectic path induced by the symplectic vector field $X_{\varepsilon\eta}$ defined by $\iota_{X_{\varepsilon\eta}}\om=\varepsilon\eta$. Then, for all $\varepsilon>0$ we have that
\begin{equation*}
	\wflux(\{\psi^{\varepsilon\eta}_{t}\}) = \varepsilon[\eta] \in \ker\{\pi^{*}: \h^{1}(M;\R) \ra \h^{1}(\ol{M}^{\eta};\R)\}.
\end{equation*}
Therefore, Theorem \ref{thm: FNH1_to_FNH2} implies that 
\begin{equation}\label{eq: rank_comparison1}
	\rank_{\til\Lambda_{\eta, \om}}\hfn^{*}(\{\psi'_{t}\},\ol{M}^{\eta}; J)=\rank_{\til\Lambda_{\varepsilon\eta, \om}}\hfn^{*}(\{\psi^{\varepsilon\eta}_{t}\},\ol{M}^{\eta}; J).
\end{equation}
Finally, Equations (\ref{eq: equality_of_complexes})-(\ref{eq: rank_comparison1}), Theorem \ref{thm: FNH_and_MNH}, and Proposition \ref{prop: Le-Ono_C.2}, respectfully, imply that for $\varepsilon>0$ sufficiently small, 
\begin{align*}
	\rank_{\Lambda_{\om}}\h^{*}(M;\Q)\otimes_{\Q}\Lambda_{\om} &= \rank_{\Lambda_{\om}}\hfn^{*}(\{\varphi_{t}\}; J)\\
		&= \rank_{\Lambda_{\eta, \om}}\hfn^{*}(\{\psi'_{t}\}, J')\\
		&= \rank_{\til\Lambda_{\eta, \om}}\hfn^{*}(\{\psi'_{t}\},\ol{M}^{\eta}; J)\\
		&=\rank_{\til\Lambda_{\varepsilon\eta, \om}}\hfn^{*}(\{\psi^{\varepsilon\eta}_{t}\},\ol{M}^{\eta}; J)\\
		&=\rank_{\Lambda_{\varepsilon\eta, \om}}\hfn^{*}(\{\psi^{\varepsilon\eta}_{t}\}; J)\\
		&=\rank_{\Lambda_{\varepsilon\eta,\om}}\hn^{*}(M, \varepsilon\eta)\otimes_{\Lambda_{\varepsilon\eta}}\Lambda_{\varepsilon\eta,\om}\\
		&=\rank_{\Lambda_{\eta,\om}}\hn^{*}(M,\eta)\otimes_{\Lambda_{\eta}}\Lambda_{\eta,\om}.
\end{align*}
Hence, $\eta$ must be exact. Indeed, if $[\eta]\neq0$, then by Proposition \ref{prop: Morse_v_MorseNovikov}
\begin{align*}
	\rank_{\Lambda_{\om}}\h^{*}(M;\Q)\otimes_{\Q}\Lambda_{\om} &=  \dim_{\Q}\h^{*}(M;\Q)\\
		&> \dim_{\Lambda_{\eta}}\hn^{*+n}(M,\eta)\\
		&= \rank_{\Lambda_{\eta,\om}}\hn^{*+n}(M,\eta)\otimes_{\Lambda_{\eta}}\Lambda_{\eta,\om},
\end{align*}
in contradiction with the above equalities. This concludes the proof of Theorem \ref{thm: main}.
\subsection{A proof using a result of McDuff}
It follows from \cite[Theorem 1]{McDuff_Monotone} that if $\{\psi_{t}\}$ is a symplectic loop, then $\cI_{c_{1}}$ vanishes on the elements of $\pi_{1}(\cL M)$ that are represented by $\{\psi_{t}(\ga(s))\}_{s,t\in[0,1]}$ for a loop $\gamma\in\cL M$. If $\{\psi_{t}\}$ has trivial evaluation, then $\{\psi_{t}(\ga(s))\}_{s,t\in[0,1]}$ determines a homotopy class $A_{\ga}\in\pi_{2}(M)$ with $\brat{c_{1}(M),A_{\ga}}=0$. If $(M,\om)$ is spherically monotone, for every $1$-cycle in $M$ given by a loop $\ga$ we have that
\begin{equation*}
	\brat{\wflux(\{\psi_{t}\}), [\ga]} = \brat{[\om], A_{\ga}}=\la\brat{c_{1}(M),A_{\ga}}=0.
\end{equation*}
Therefore, $\wflux(\{\psi_{t}\})=0$, which yields once again the conclusion of Theorem \ref{thm: main}. In addition, note that if $\{\psi_{t}\}$ has trivial evaluation and flux $[\eta]\neq0$ we can produce a homotopy class $\alpha\in\pi_{1}(\cL M)$ such that $\al\in\ker(\cI_{\om}+\cI_{\eta})\cap\ker\cI_{c_{1}}$ while $\cI_{\eta}(\al)=-\cI_{\om}(\al)\neq0$. Indeed the class represented by the loop-of-loops $\al(s,t)=\psi_{t}(\ga(s))$ satisfies these properties. We are then able to conclude the following:

\begin{cor}
If the $e$-homomorphism is not injective, then \[\ker\cI_{\om}\cap\ker\cI_{\eta}\cap\ker\cI_{c_{1}}\subsetneq\ker(\cI_{\om}+\cI_{\eta})\cap\ker\cI_{c_{1}}\] for all $[\eta]\in\ker e$.
\end{cor}

\noindent This shows, in hindsight, why the equality of Novikov rings in the spherically monotone setting yielded a proof of injectivity of the $e$-homomorphism.

\section{Acknowledgements}
\noindent This work is part of the author's Ph.D. thesis carried out at the Université de Montréal under the supervision of Egor Shelukhin, who we thank for the inspiring interactions and countless suggestions. We thank Dustin Connery-Grigg for the suggestion to re-examine McDuff's result \cite{McDuff_Monotone}, which yielded an alternative proof of the main theorem. We also thank Deeparaj Bhat for insightful conversations concerning the Appendix. This research was partially supported by Fondation Courtois. This material is based upon work supported by the National Science Foundation under Grant No. DMS-1928930, while the author was in residence at the Simons Laufer Mathematical Sciences Institute (previously known as MSRI) Berkeley, California during the Fall 2022 semester.
	
\section{Appendix}
\noindent In the process of studying fixed points of symplectic torsion we proved the following proposition, which we believe was known to experts but decided to put it in writing for the sake of completeness.

\begin{prop}\label{prop: appendix}
Let $M$ be a closed smooth manifold and $\varphi$ a non-trivial orientation preserving diffeomorphism of finite order. If $\#\fix(\varphi)<\infty$, then
	\begin{equation*}
		\#\fix(\varphi) = \sum_{k\geq0}(-1)^{k}\mrm{tr}\big(\varphi_{*}|_{\h_{k}(M;\Q)}\big).
	\end{equation*}
In particular, if $\varphi$ acts trivially on homology, we have 
	\begin{equation*}
		\#\fix(\varphi) = \chi(M).	
	\end{equation*}
 \end{prop}
\begin{proof}
The fact that $\varphi$ is a diffeomorphism of finite order implies, in general, that $\fix(\varphi)$ is given by a disjoint union of isolated smooth submanifolds $F_{1},\cdots, F_{b}$ satisfying $\mrm{det}(D(\varphi)_{x}-\id_{T_{x}M}) = T_{x}F_{j}$ for all $x\in F_{j}$, $1\leq j\leq b$. Since there are only finitely many fixed points by assumption, we find that $\varphi$ is non-degenerate. Furthermore, for each $x$, let $A_{x}\in\mrm{GL}(n,\C)$ correspond to the Jordan form of $D(\varphi)_{x}$. It is a diagonal matrix since $A_{x}$ is of finite order. The fact $\varphi$ is orientation preserving and non-degenerate implies that the diagonal of $A_{x}$ consists of conjugate pairs of roots of unity $\lambda_{i}$ (possibly including pairs of $-1$). Hence,
\begin{equation*}
	\mrm{det}(A_{x}-\id) = \prod_{j=1}^{k} (\lambda_{j}-1)(\ol\lambda_{j}-1) = \prod_{j=1}^{k} |\lambda_{j}-1|^{2}>0.
\end{equation*}
Since $\varphi$ is a non-degenerate diffeomorphism with only finitely many fixed points, the Lefschetz-Hopf theorem implies that
	\begin{equation*}
		\#\fix(\varphi)=\sum_{x\in\fix(\varphi)}\sign(\mrm{det}(D(\varphi)_{x}-\id_{T_{x}M}))= \sum_{k\geq0}(-1)^{k}\mrm{tr}\big(\varphi_{*}|_{\h_{k}(M;\Q)}\big).	
	\end{equation*}
\end{proof}

\bibliographystyle{amsplain}
\bibliography{bibliographySCA}

\end{document}